\documentclass[12pt,letter]{article}
\usepackage{amsmath}
\usepackage{ytableau}
\usepackage{amsthm,stmaryrd}
\usepackage{amsfonts, mathrsfs,amssymb,hyperref,pdfsync}
\usepackage{float}
\usepackage{amsmath,bbm}
\usepackage[numbers, sort&compress]{natbib}
\usepackage{vmargin, enumerate}
\usepackage[utf8]{inputenc}
\usepackage[normalem]{ulem}
\usepackage{pgf, tikz}
\usetikzlibrary{trees}
\usetikzlibrary{patterns}
\usepackage{svg}
\usepackage{tikz}
\usetikzlibrary{automata,arrows}
\usepackage{graphicx}
\usepackage{lipsum}
\usepackage{wrapfig}
\usepackage{caption}

\newtheorem{theo}{Theorem}
\newtheorem{prop}{Proposition}

\theoremstyle{definition}

\newtheorem{rem}{Remark}

\newcommand{\euler}[2]{\left\langle\begin{array}{c}
     \!\!#1\!\! \\
     \!\!#2\!\!
\end{array}\right\rangle}

\newcommand{\ind}{ 1\hspace{-.55ex}\mbox{l}}

\newcommand{\impl}{\zeta}

\newcommand{\N}{\ensuremath{\mathbb N}}

\newcommand{\ac}[1]{\left\{#1\right\}}
\newcommand{\pa}[1]{\left(#1\right)}

\newcommand{\abs}[1]{\left|#1\right|}

\newcommand{\pr}[1]{\mathbb P\left(#1\right)}

\begin{document}\title{\bf  Pascal's formulas and vector fields}

\author{Philippe Chassaing \thanks{ Institut Élie Cartan, Université de Lorraine Email: chassaingph@gmail.com}, Jules Flin \thanks{Institut Élie Cartan, Université de Lorraine Email: jules.flin8@etu.univ-lorraine.fr}, Alexis Zevio \thanks{Institut Élie Cartan, Université de Lorraine Email: alexis.zevio1@etu.univ-lorraine.fr }}

\maketitle
\begin{abstract} \noindent We consider four examples of combinatorial triangles $\pa{T(n,k)}_{0\le k\le n}$ (Pascal, Stirling of both types, Euler) : through saddle-point asymptotics, their  \emph{Pascal's formulas} define four vector fields, together with their field lines that turn out to be the conjectured limit of sample paths of four well known Markov chains. We  prove  this asymptotic behaviour in three of the four cases.
\end{abstract}

\tableofcontents

\section{Introduction}

\subsection{Pascal's formulas} 
\label{sec:intro}
Set $S=\ac{(n,k)\in\N^2,\ 0\le k\le n}$, $\mathring{S}=\ac{(n,k)\in\N^2,\ 0< k< n}$, and let $S^{\star}=S\backslash\ac{\pa{0,0}}$.
Besides Pascal's triangle, other triangular arrays $\pa{T(n,k)}_{(n,k)\in S}$ of interest satisfy a recursion formula similar to Pascal's formula, i.e. of the following form, for $(n,k)\in S^{\star}$ :
\begin{align}
\label{Pascal}
T(n,k)&=a(n,k)T(n-1,k-1)+b(n,k)T(n-1,k),
\end{align}
with the convention that either $(n,k)\in S$ or $T(n,k)=0$. For instance, relation \eqref{Pascal} holds true for the following triangular arrays :
\begin{itemize}
\item for Pascal's triangle, if $(a,b)(n,k)=(1,1)$ ; 
\item for Stirling numbers of the second kind, if $(a,b)(n,k)=(1,k)$ ; 
\item for Stirling numbers of the first kind, if $(a,b)(n,k)=(1,n-1)$ ; 
\item for Euler's triangle, if $(a,b)(n,k)=(n-k,k+1)$. 
\end{itemize}
\subsection{Transition probabilities} 
\label{sec:mainr}
In view of \eqref{Pascal}, for $(n,k)\in S^{\star}$, consider
\begin{align}
\label{pt}
\pa{p_{0}(n,k),p_{1}(n,k)}&= \pa{\frac{b(n,k)T(n-1,k)}{T(n,k)}, \frac{a(n,k)T(n-1,k-1)}{T(n,k)}}
\end{align}
as some transition probabilities from $(n,k)$ to $(n-1,k)$, resp. to $(n-1,k-1)$. For each of these four triangular arrays, the transition probabilities $$\pa{p_\varepsilon(n,k)}_{(\varepsilon,(n,k))\in\{0,1\}\times S^{\star}},$$ together with the initial state $(m,\ell)$,  define a Markov chain $W=\pa{W_{k}}_{0\le k\le n}$ with terminal state  $(0,0)$. These four Markov chains are closely related to the simple random walk, the coupon collector problem, the chinese restaurant process and the one-dimensional internal DLA, respectively : they are the time-reversed versions of these processes, once these processes are conditioned to be at level $\ell$ at time $m$, as explained in the next section.

\subsection{Random walk, coupon collector, chinese restaurant,  and internal DLA} 
\label{examplesMC}
Consider a random process defined by $X_{0}=0$, and, for $n\geqslant 0$, $X_{n+1}= X_{n}+Y_{n+1},$ in which the $Y_{i}$'s are Bernoulli random variables. Set
\begin{align*}
W_{n}&=(m-n,X_{m-n})\in S,\quad 0\le n\le m, \\
w_{m}(t)&=m^{-1}\,X_{\lfloor mt\rfloor}\in S,\quad 0\le t\le 1,
\end{align*}
and note that, by definition, $W_{n}=(0,0)$ if and only if $n=m$. 

\subsubsection{Simple random walk}
Assume that  $(Y_{i})_{i\ge1}$ is a Bernoulli process, i.e. a sequence of i.i.d. Bernoulli random variables with parameter $p\in(0,1)$. Then
\begin{prop}
\label{kennedy}
The stochastic process $W=(W_{n})_{0\le n\le m}$, conditioned to $W_{0}=(m,\ell)$, or equivalently to $X_{m}=\ell$, is the Markov chain with transition probabilities $(p_{\varepsilon}(n,k))_{(\varepsilon,n,k)\in\ac{0,1}\times S^{\star}}$ related to Pascal's triangle. Its distribution does not depend on $p$.
\end{prop}
This result goes back at least to Kennedy \cite{zbMATH03504238}, or even to the introduction of the concept of sufficiency by Fisher around 1920 \cite{zbMATH03448510}. We recall its proof at Subsection \ref{kennedytype}. In the next cases, the Bernoulli random variables $Y_{i}$ are not i.i.d. .

\subsubsection{Coupon collector's problem}
Consider the coupon collector's problem with $N$ different items. Let $X_{n}$ denote the number of different items in the collection after the $n$th step. Again :
\begin{prop}
\label{coupon}
The stochastic process $W=(W_{n})_{0\le n\le m}$, conditioned on $W_{0}=(m,\ell)$, or equivalently on the number of different items in the collection after the $m$th step, $X_{m}$, to be equal to $\ell$, 
is the Markov chain with transition probabilities $(p_{\varepsilon}(n,k))_{\varepsilon,n,k}$ related to Stirling numbers of the second kind. Its distribution does not depend on $N$.
\end{prop}

\subsubsection{Chinese restaurant process} In the Chinese restaurant process with $(0,\theta)$ seating plan, defined at Section \ref{kennedytype} (see also, e.g., \cite[Ch. 3]{MR2245368}), let $X_{n}$ denote the number of occupied tables after the arrival of the $n$th customer. 
\begin{prop}
\label{chine}
The stochastic process $W=(W_{n})_{0\le n\le m}$, conditioned to $W_{0}=(m,\ell)$, or equivalently to $X_{m}=\ell$, 
is the Markov chain with transition probabilities $(p_{\varepsilon}(n,k))_{\varepsilon,n,k}$ related to Stirling numbers of the first kind. Its distribution does not depend on $\theta$.
\end{prop}
\begin{rem}
As a consequence, in the three previous cases, given the data $(X_{n})_{0\le n\le m}$,  $X_m$ (or $W_0$) are sufficient statistics for the parameters $p$, $N$ or $\theta$, respectively.
\end{rem}

\subsubsection{One-dimensional Internal Diffusion Limited Aggregation} 
Finally, in the one-dimensional Internal Diffusion Limited Aggregation process (iDLA), let $X_{n}$ denote the number of particles settled to the right of the origin after the release of the $n$th particle. Then
\begin{prop}
\label{idla}
The stochastic process $W=(W_{n})_{0\le n\le m}$, conditioned to $W_{0}=(m,\ell)$, is the Markov chain with transition probabilities $(p_{\varepsilon}(n,k))_{\varepsilon,n,k}$ related to Euler's triangle.
\end{prop}
More precise definitions, and proofs, are to be found at Section \ref{kennedytype}.
\subsection{Simulations}
\label{simulations}

In the case of  Pascal's triangle, the behaviour of this time-reversed Markov chain is well understood since forever.  Quite recently, \cite{amri:hal-02164935} gave a rather precise analysis of the analog time-reversed Markov chain related to Stirling numbers of the second kind, with combinatorial analysis of finite automata as a motivation. We hope to improve some of their results and proofs.\\

In this section, in order to surmise the behaviour of these time-reversed Markov chains,  we present the result of some  simulations. For each case, the figures below show sample paths starting at $(m,mt)$ with $t\in \{0.05,...,0.95\}$ and $m=500$ :\\[-1cm]

\newpage
\begin{center}
\begin{minipage}{\linewidth}
    \centering
    \begin{figure}[H]
    \includegraphics[width=\linewidth]{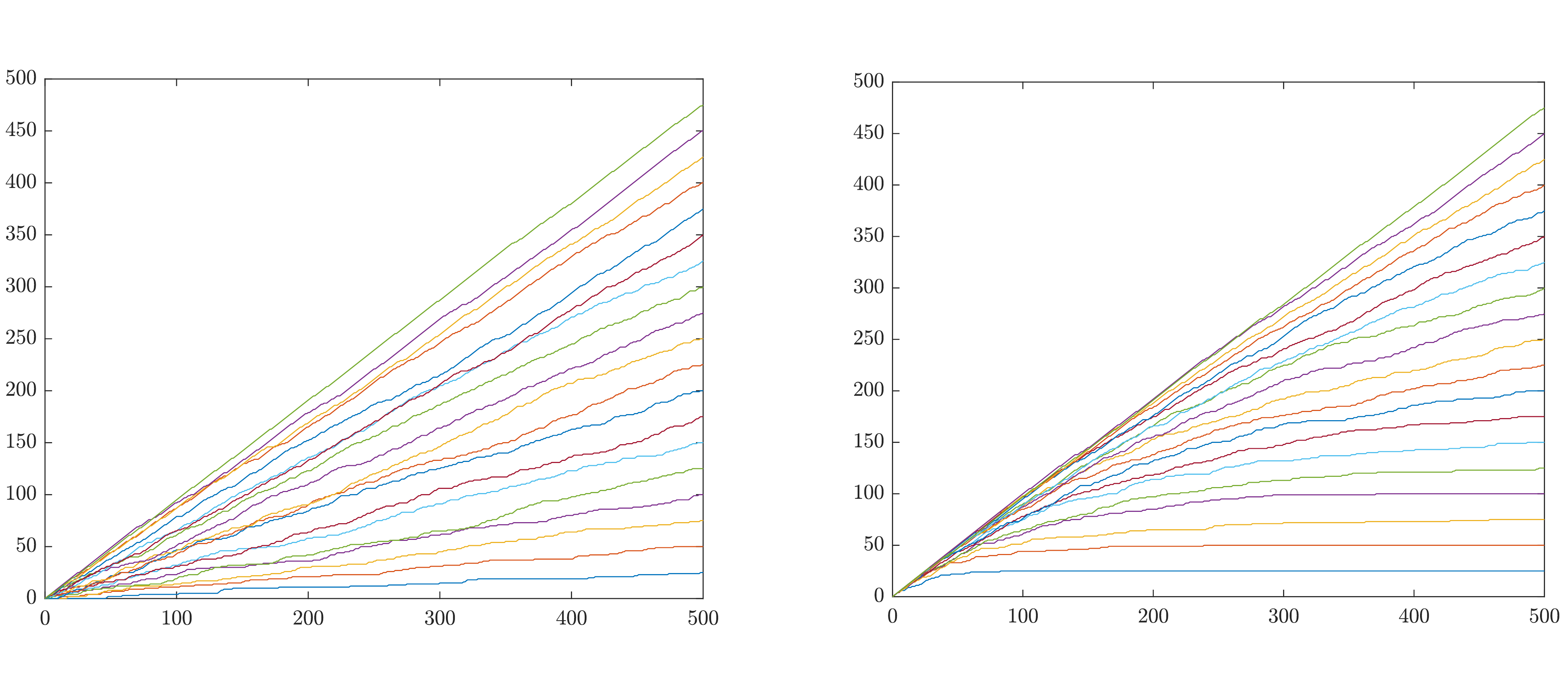}
    \begin{minipage}{0.45\linewidth}
    \centering
          Figure 1: Pascal's triangle.
      \end{minipage}
      \hspace{0.05\linewidth}
      \begin{minipage}{0.45\linewidth}
      \centering
          Figure 2: Stirling numbers of the second kind.
      \end{minipage}
    \end{figure}
\end{minipage}
\begin{minipage}{\linewidth}
    \centering
    \begin{figure}[H]
    \includegraphics[width=\linewidth]{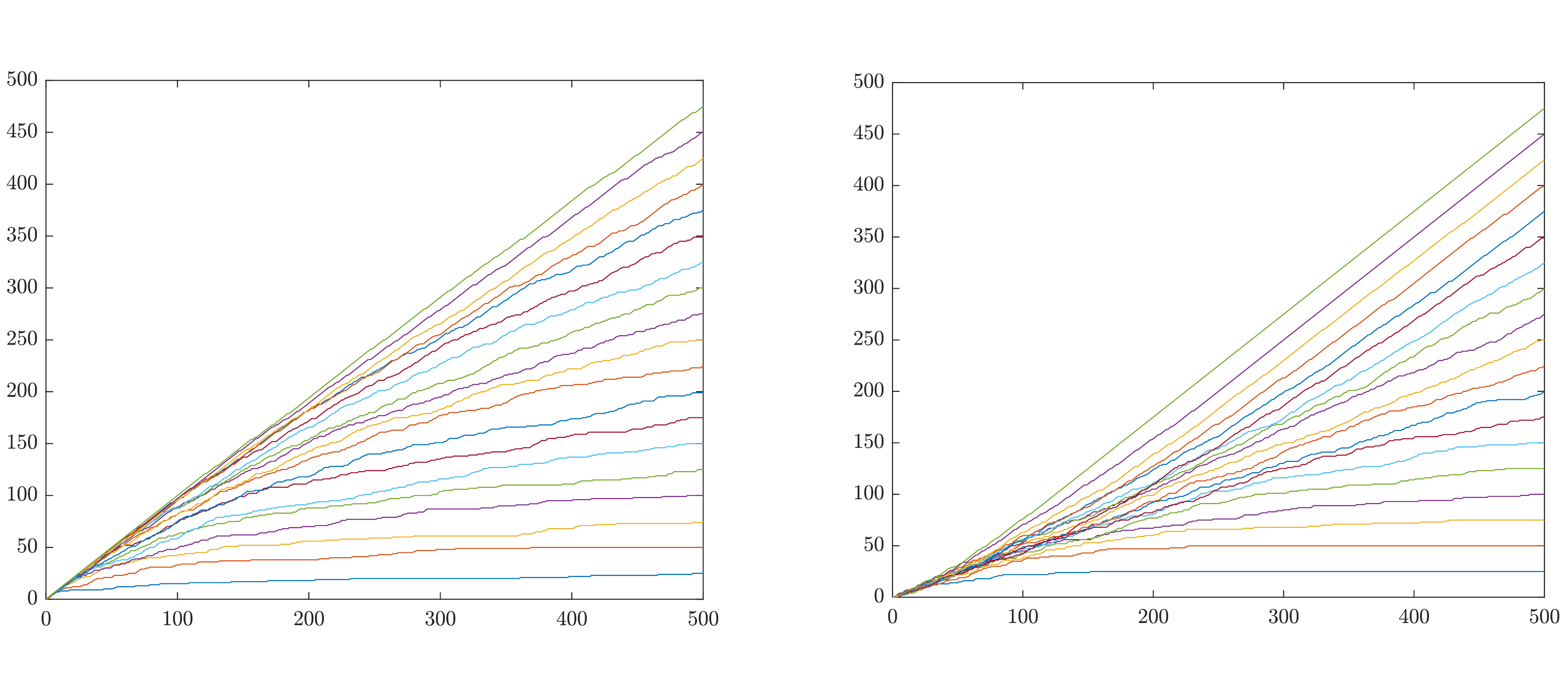}
    \begin{minipage}{0.45\linewidth}
    \centering
          Figure 3: Stirling numbers of the first kind.
      \end{minipage}
      \hspace{0.05\linewidth}
      \begin{minipage}{0.45\linewidth}
      \centering
        Figure 4: Eulerian numbers.\\
      \end{minipage}
    \end{figure}
\end{minipage}
\end{center}

Now, in order to compare the four combinatorial triangles, we show the average of $100$ sample paths for each triangle, for $m=1000$ and $t\in \{0.05,...,0.95\}$ :

\begin{center}
\begin{minipage}{\linewidth}
    \centering
    \begin{figure}[H]
    \includegraphics[width=\linewidth]{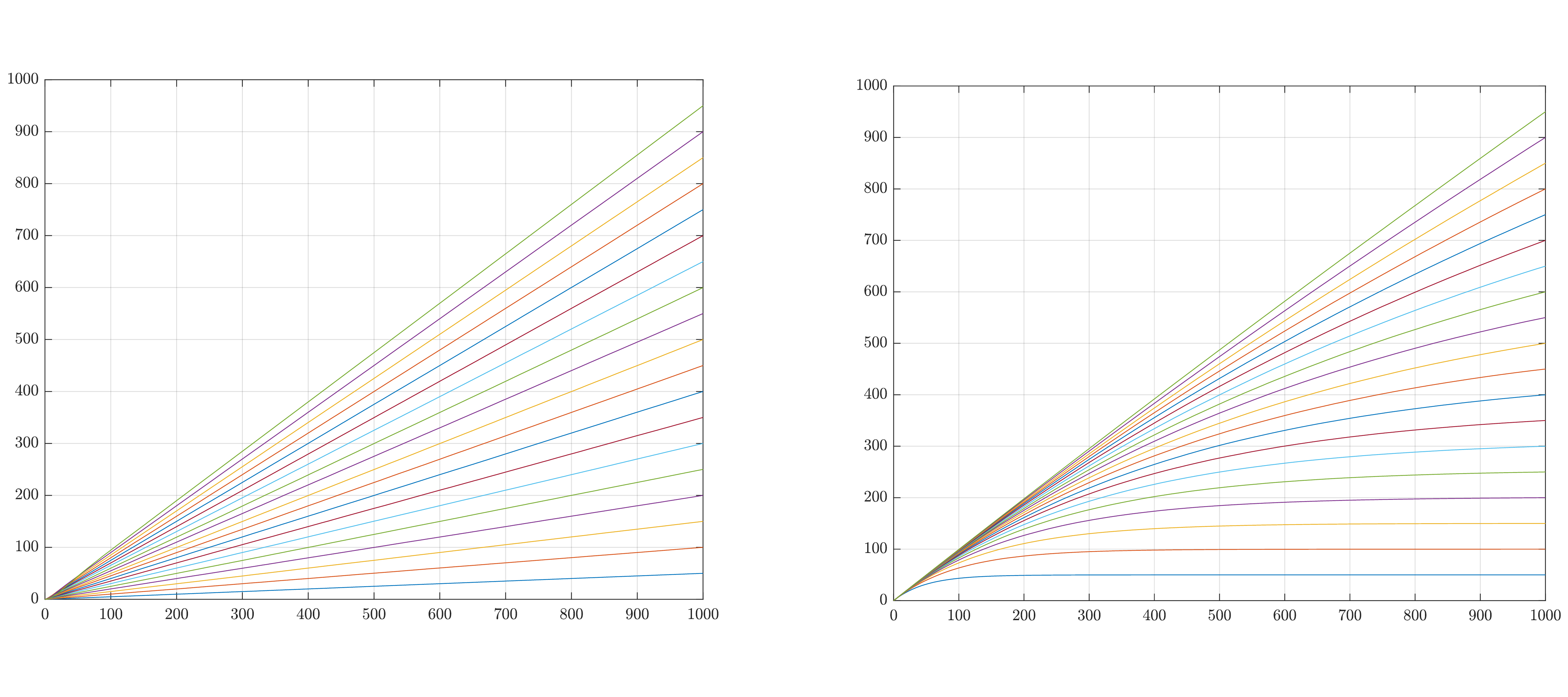}
    \begin{minipage}{0.45\linewidth}
    \centering
          Figure 5: Pascal's triangle.
      \end{minipage}
      \hspace{0.05\linewidth}
      \begin{minipage}{0.45\linewidth}
      \centering
          Figure 6: Stirling numbers of the second kind.
      \end{minipage}
    \end{figure}
\end{minipage}
\begin{minipage}{\linewidth}
    \centering
    \begin{figure}[H]
    \includegraphics[width=\linewidth]{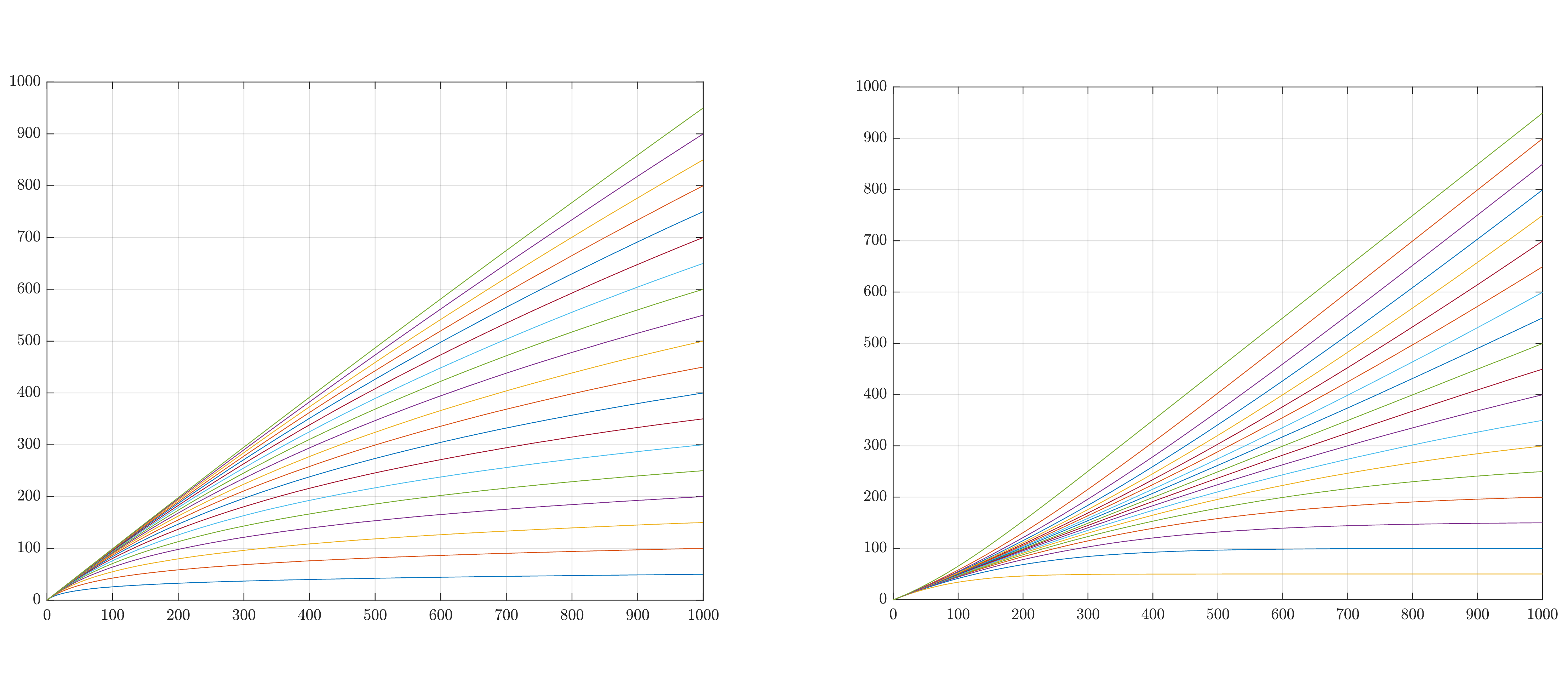}
    \begin{minipage}{0.45\linewidth}
    \centering
          Figure 7: Stirling numbers of the first kind.
      \end{minipage}
      \hspace{0.05\linewidth}
      \begin{minipage}{0.4\linewidth}
      \centering
        Figure 8: Eulerian numbers.\\
      \end{minipage}
    \end{figure}
\end{minipage}
\end{center}

In the first two cases, the smooth nature of these averaged paths is not unexpected due to old, and more recent, fluid approximation results, see \cite{amri:hal-02164935} or the next sections. This paper aims at a global explanation of the asymptotic behaviour of the four Markov chains exhibited by these simulations. 


\subsection{Asymptotics of sample paths, and field lines of vector fields}
\label{sec:stirling}
Combinatorial analysis, see \cite{MR0120204} or \cite{MR375433}, yields that, 
\begin{theo}
\label{slope}
In each of the four cases, there exists a function $\varphi : (0,+\infty)\rightarrow[0,1]$ such that,  for any positive number $\lambda_{\infty}$, when $(m,\ell)\rightarrow+\infty$ and $\lim m/\ell=1+\lambda_{\infty}$, 
\[\lim p_{1}(m,\ell)=\varphi(\lambda_{\infty}).\]
\end{theo}
\noindent At the end of this section,  the function $\varphi$ is described for each of the four cases. 

We set, for $(m,\ell)\in\mathring{S}$,
$$\lambda=\lambda(m,\ell)=\frac{m-\ell}\ell.$$
As a direct consequence of Theorem \ref{slope}, one expects a fluid approximation of the previous Markov chains by a special family of curves : let $\mathbb{P}_{(m,\ell)}$ denote the probability distribution of the Markov chain $W$ starting from $(m,\ell)$, and let $(x,\gamma_{\lambda}(x))_{0\le x\le n}$ be the field line going through the point $(1,\ell/m)=\pa{1,1/(1+\lambda)}$ for the vector field $(1,\varphi(-1+x/y))$, or, equivalently, let $\gamma_{\lambda} $ denote the solution of the ODE  
\begin{equation}
\label{fieldline}
y^\prime=\varphi\pa{\frac {x-y}y}
\end{equation}
that satisfies $y(1)=1/(1+\lambda)$. So far we have a complete proof of this approximation  only in the first three cases :
\begin{theo}
\label{samplepath}
In the first three cases, for any $\eta\in(0,1/2)$ and any $\lambda_{\infty}>0$, when $(m,\ell)\rightarrow+\infty$ and $\lim \lambda(m,\ell)=\lambda_{\infty}$, 
\[
\lim\mathbb{P}_{(m,\ell)}\pa{\sup_{0\le t\le 1}\pa{\abs{w_m(t)-\gamma_{\lambda_{\infty}}(t)}}\ge m^{-\eta}}=0.
\]
\end{theo}

\noindent Note that the special form of the ODE \eqref{fieldline} entails that the set of field lines is invariant by positive homotheties.

This kind of statement seems to hold true for eulerian numbers, according to our simulations (see Section \ref{simulations}), but remains an open question. For  Stirling numbers of the first kind, Theorem \ref{samplepath} seems to be new, as far as we know. For Stirling numbers of the second kind, Theorem \ref{samplepath}  is a vastly improved version of a result that appeared in \cite{amri:hal-02164935}, in which the  proof relies mainly on  uniform bounds for 
\[m\ \abs{p_{1}(m,\ell)-\varphi\pa{\lambda(m,\ell)}},\]
on domains that approach $\mathring{S}$ as well as possible. These bounds  follow from a careful asymptotic analysis of $T(m,\ell)$, that should have some interest in itself. However the proof given here is much simpler.

Our choice of combinatorial triangles may seem arbitrary, and we confess it is : for instance, Bell's triangle or Delannoy's triangle have also Pascal's formulas, but of a slightly different form. It remains to see if the approach of this paper still produces results for Bell's triangle or Delannoy's triangle, in spite of these slight differences.

\subsubsection{Description of $\varphi$}
\begin{itemize}
\item \emph{Pascal's triangle.} It is well known that for all $(m,\ell)\in S^{\star}$,
$$p_{1}(m,\ell)=\frac {\ell}m,$$
so that 
$$\varphi_{1}(\lambda)=\frac1{1+\lambda}.$$
Relation \eqref{fieldline} reduces to $y^\prime=y/x$, with the linear functions as solutions, as expected.
\item \emph{Stirling numbers of the second kind.} For $\lambda>0$, let $\varphi_{2}$ be defined, through $\impl_{2}(\lambda)$, the unique positive solution of
\begin{equation}
\label{ckor2}
\frac{\impl_{2}}{1-e^{-\impl_{2}}}=1+\lambda,
\quad\text{by}\ \varphi_{2}(\lambda)=e^{-\impl_{2}}.
\end{equation}
Then, for $x\ge 0$,
\begin{equation}
\label{gd1}
\gamma_{\lambda}(x)=\frac{1-e^{-x\,\impl_{2}(\lambda)}}{\impl_{2}(\lambda)}.
\end{equation}

\item \emph{Stirling numbers of the first kind.} For $\lambda>0$, let $\varphi_{3}$ be defined, through $\impl_{3}(\lambda)$, the unique  solution, in $(0,1)$,  of
$$\frac{\impl_{3}}{\pa{\impl_{3}-1}\ln(1-\impl_{3})}=1+\lambda,
\quad\text{by}\ \varphi_{3}(\lambda)=1-\impl_{3}.$$
Then, for $x\ge 0$,
\begin{equation}
\label{gd2}
\gamma_{\lambda}(x)=\frac{1-\impl_{3}(\lambda)}{\impl_{3}(\lambda)}\ln\pa{\frac{1-\impl_{3}(\lambda)+x\,\impl_{3}(\lambda)}{1-\impl_{3}(\lambda)}}.
\end{equation}

\item \emph{Eulerian numbers.} For $\lambda>0$, let $\varphi_{4}$ be defined, through $\impl_{4}(\lambda)$, the unique  solution, in $\mathbb{R}$,  of
$$\frac1{1+\lambda}=\frac{e^{\impl_{4}}}{e^{\impl_{4}} -1}-\frac{1}{\impl_{4}},
\quad\text{by}\ \varphi_{4}(\lambda)=1-\frac{\impl_{4}}{(1+\lambda)(e^{\impl_{4}}-1)}.$$
At the moment, we are unaware of any closed form formula for $\gamma_{\lambda}$ in this case.
\end{itemize}

\section{Time reversal and Markov property for $W$}
\label{kennedytype}
In this section, we prove Propositions \ref{kennedy}, \ref{coupon}, \ref{chine} and \ref{idla}. The notations  $X_{n}$, $Y_{n}$, $W_{n}$ are defined at Section \ref{examplesMC}.

\subsection{Time reversal}\label{sec:timereversal}
As already known at least since Kolmogorov, see \cite[(7)]{zbMATH03019588}, a time-reversed Markov process is still a Markov process, but it is an inhomogeneous one. Let us recall the basic facts that we need here : if $h_k$ denotes the probability distribution of $X_k$ and if $X=(X_k)_{k\ge0}$ is an inhomogeneous Markov chain with kernels $(Q_k) _{k\ge0} $, i.e.
$$Q_{k,i,j}=\pr{X_{k+1}=j\ |\,X_k=i},$$
then 
\begin{prop}
\label{timereversal}
$W=(W_n)_{0\le n\le m}$ is a Markov chain with state space $S$ and with kernel $P$ defined on $S^\star$ by 
$$P_{(n,i),(n-1,j)}=\frac{h_{n-1}(j)Q_{n-1,j,i}}{h_{n}(i)}.$$
\end{prop}
\begin{proof}First, since $(0,0)$ is only reached, eventually, at time $m$, there is no need to define $P_{(0,0),(.,.)}$. Also, $P$ is a  probability kernel due the Chapman-Kolmogorov equations for $(h_n)$ and $(Q_n)$. Then, 
\begin{align*}
\pr{(X_k)_{0\le k\le m}=(x_k)_{0\le k\le m}}&=h_0(x_0)\prod_{k=0}^{m-1}Q_{k,x_k,x_{k+1}},
\end{align*} 
thus, provided that $x_m=\ell$,
\begin{align*}
\pr{(X_k)_{0\le k\le m}=(x_k)_{0\le k\le m}\,|\ X_m=\ell}
&=h_0(x_0)\prod_{k=0}^{m-1}Q_{k,x_k,x_{k+1}}\ /\, h_m(\ell),
\\
&= \prod_{k=0}^{m-1} P_{(k+1,x_{k+1}),(k,x_k)}.
\end{align*} 
That is,
\begin{align*}
\pr{(W_k)_{0\le k\le m}=(m-k,x_{m-k})_{0\le k\le m}\,|\ W_0=(m,\ell)}
&= \prod_{k=0}^{m-1} P_{(k+1,x_{k+1}),(k,x_k)}, 
\end{align*} 
as expected.
\end{proof}
But for eulerian numbers, $h_n(k)=T(n,k)\theta^k/T_n(\theta)$, or $h_n(k)=T(n,k)\theta^{k\downarrow}/T_n(\theta)$, in which $T_n(\theta)$ is a normalizing constant :
\begin{equation}
T_n(\theta)=\sum_{k=0}^{n}T(n,k)\theta^k,\quad\textrm{or}\quad T_n(\theta)=\sum_{k=0}^{n}T(n,k)\theta^{k\downarrow}.
\end{equation}
For eulerian numbers, $h_n(k)=T(n,k)/T_n(1)=T(n,k)/n!$.

Note that $Q_n$ results from a natural growing mechanism with independent steps, that is, a Markovian growth process, obtained as follows : 
\begin{itemize}
\item 
by addition of an $n+1$th letter, either $\texttt{a}$ or $\texttt{b}$,  at the end of a random word of $\ac{\texttt{a}, \texttt{b}}^n$, in order to form an $n+1$-letters long word, for Pascal's triangle, 
\item by addition of the image of $n+1$ to a random mapping from  $[\![n]\!]$ to  $[\![N]\!]$, in order to form a random mapping from  $[\![n+1]\!]$ to  $[\![N]\!]$, for the the second Stirling triangle, 
\item by random insertion of $n+1$ in order to form a permutation on $[\![n+1]\!]$, starting from a permutation on $[\![n]\!]$, for the 2 other examples. 
\end{itemize}
In each case, the added letter, or integer, is chosen independently of the previous history of the growth process, hence the Markovian character of these growth processes. 
For the sake of brevity, in this paper, we call the last growth process the \emph{random permutation process}. For these three nonhomogeneous Markov growth processes, there exist well studied functionals that retain the Markov property, and whose one-dimensional distributions are given by the rows of the corresponding combinatorial triangle :
\begin{itemize}
\item the sequence of counts of letter $\texttt{a}$, in the sequence of words  defined previously, forms one of the most studied Markov chain : the simple random walk, whose one-dimensional distributions $h_n$ are binomial distributions, famously related to Pascal's triangle ;
\item the sequence of sizes of images, derived from the sequence of random mappings, is a famous inhomogeneous Markov chain, related to the coupon collector problem : it is the sequence of successive sizes of the collection. Its one-dimensional distributions $h_n$ have a simple expression in terms of the Stirling numbers of the second kind ;
\item the sequence of number of cycles, derived from the random permutation process, is an inhomogeneous Markov chain, related to the chinese restaurant process. Its one-dimensional distributions $h_n$ have a simple expression in terms of the Stirling numbers of the first kind ;
\item  the sequence of the number of descents, also derived from the random permutation process, is an inhomogeneous Markov chain, related to the internal diffusion limited aggregation process. Its one-dimensional distributions $h_n$ have a simple expression in terms of eulerian numbers.
\end{itemize} 
Chapman-Kolmogorov equations for these Markov chains are derived from Pascal's formulas for corresponding  triangles through renormalization : in our settings, $Q_{n}$ is defined by $(a,b)$ as follows 
\begin{align}
\label{directkernel}
Q_{n,x,y}&=c_{n}(\theta)\,\pa{b(n+1,y)\ind_{y=x\in [\![n]\!]}+a(n+1,y)\ \theta\ \ind_{y=x+1\in[\![n+1]\!]}},
\end{align}
in which $c_{n} (\theta) $ denotes a normalizing factor $T_{n}(\theta)/T_{n+1}(\theta)$, and $\theta= \theta^{x+1}/\theta^x$ should be replaced, in the last factor of \eqref{directkernel}, with $\theta-x= \theta^{x+1\downarrow}/\theta ^{x\downarrow} $ in the case of Stirling numbers of the second kind. For eulerian numbers, $\theta=1$. Then,  Pascal's formulas appear as special cases of  the Chapman-Kolmogorov equation $h_{n-1}Q_{n-1}= h_{n}$, and relation \eqref{pt} is just a special case of Proposition \ref{timereversal}.

Here, $Q_{n,x_n,x_{n+1}}\neq 0$ only if $\varepsilon_{n+1}= x_{n+1}-x_{n}$ belongs to $\{0,1\}$, thus $P_{(n,x),(n-1,y)}\neq 0$ only if $\varepsilon= x-y$ belongs to $\{0,1\}$ : in this paper, $P_{(n,x),(n-1,x-\varepsilon)}$ is abridged to $p_{\varepsilon}(n,x)$.

\subsection{Simple random walk}
\begin{proof}[Proof of Proposition \ref{kennedy}]
Here 
\begin{align*}
T(n,k)&={n\choose k},\quad \theta=\frac{p}{1-p},\quad T_n(\theta)=\pa{1+\theta}^n
\\
h_n(k)&= {n\choose k}p^k(1-p)^{n-k},\quad c_n (\theta) =\frac{1}{1+\theta}=1-p.
\end{align*}
Then, for instance,
\begin{align*} 
P_{(n,i),(n-1,i-1)}&=\frac{h_{n-1}(i-1)Q_{n-1,i-1,i}}{h_{n}(i)}
\\
&= \frac{ {n-1\choose i-1}p^{i-1}(1-p)^{n-i}\times p	}{{n\choose i} p^{i}(1-p)^{n-i}}= \frac{ {n-1\choose i-1}}{{n\choose i}}= p_{1}(n,i)
\end{align*}
as expected.
\end{proof}

\subsection{Coupon collector's problem}
Let us recall the famous problem studied by Gauss and Laplace, among others : a collector wants to complete a collection of $N$ different items (denoted $1,...,N$). At each step, he receives a coupon chosen uniformly from $[\![1,N]\!]$.  The average time to complete the collection is known to be $NH_N$, where $$H_N=\sum_{k=1}^N \frac{1}{k}$$ 
is the $N$th harmonic number. If  $X_{n}$ denotes the number of different items in the collection after the $n$th step, then we call the graph of  $t\mapsto X_{\lfloor t \rfloor}$ the \textit{completion curve}.
\begin{center}
    \includegraphics[width=\linewidth]{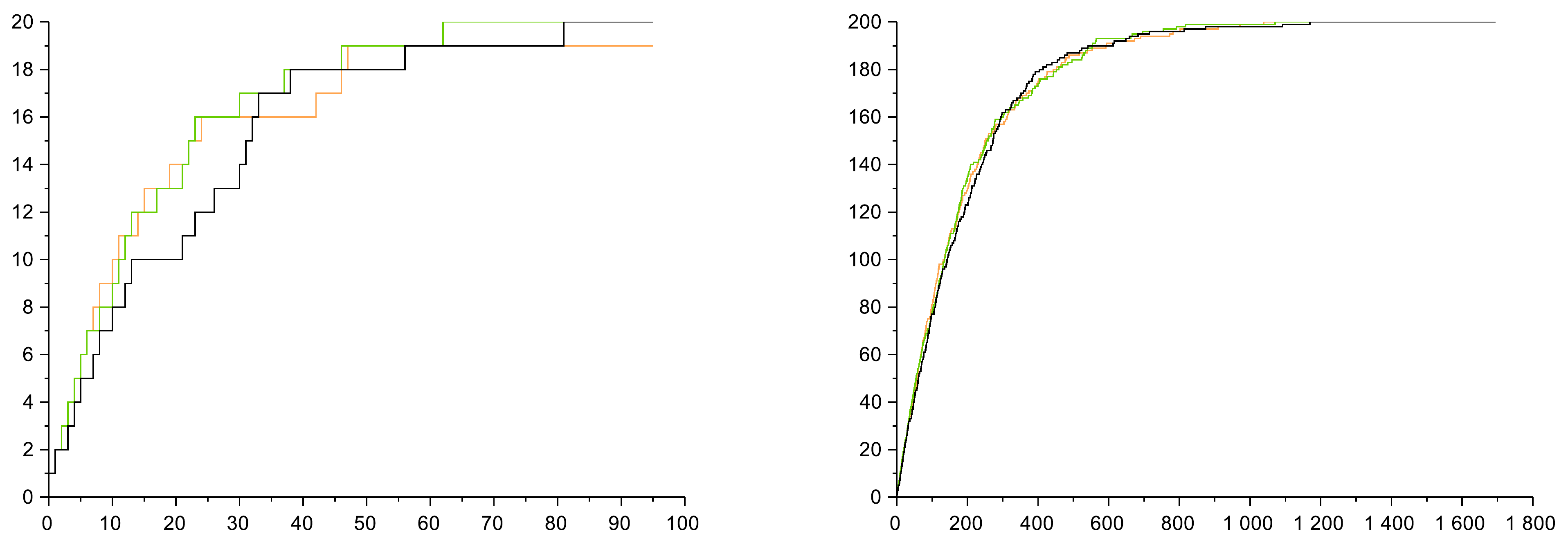}\\[0.5cm]
    \begin{minipage}{\linewidth}
    \centering
    \begin{minipage}{0.45\linewidth}
    \centering
          Figure 9: Three completion curves for a $n=20$ items collection.
      \end{minipage}
      \hspace{0.05\linewidth}
      \begin{minipage}{0.45\linewidth}
      \centering
        Figure 10: Three completion curves for a $n=200$ items collection.
        \end{minipage}
\end{minipage}
\phantom{x}
\end{center}
\begin{proof}[Proof of Proposition \ref{coupon}]
See \cite[Proposition 1]{amri:hal-02164935}, in which the proof is given for $m=N$. It fits with the frame given at Section \ref{sec:timereversal} as follows : set
\begin{align*}
T(n,k)&={n\brace k},\quad \theta=N,
\end{align*}
but consider a variant of $T_n$. Here :
\begin{align*}
T_n(\theta)&=\sum_kT(n,k)\theta^{k\downarrow}=N^n
\\
h_n(k)&= {N\choose k}\,{n\brace k}\,k!\,\frac{1}{N^n}= {n\brace k}\frac{N^{k\downarrow}}{N^n},\quad c_n (\theta) =\frac1{N}.
\end{align*}
Then, for instance,
\begin{align*}
Q_{n,k,k+1}&=a(n+1,k) \frac{\theta^{k+1\downarrow}}{\theta^{k\downarrow}}c_n(\theta)=\frac{N-k}{N}.
\end{align*}
and
\begin{align*} 
p_{1}(n,k)&=P_{(n,k),(n-1,k-1)}=\frac{h_{n-1}(k-1)Q_{n-1,k-1,k}}{h_{n}(k)}
\\
&= \frac{ {n-1\brace k-1}\frac{N^{k-1\downarrow}}{N^{n-1}}\times \frac{N-k+1}{N}	}{{n\brace k}\frac{N^{k\downarrow}}{N^n}} = \frac{ {n-1\brace k-1}	}{{n\brace k}},
\end{align*}
as expected.
\end{proof}

\subsection{Chinese restaurant process}
\label{chinese}

Set $\theta\in (0,+\infty)$. The chinese restaurant process, introduced in 1974 by Antoniak in \cite{MR365969}, is defined as follows : when entering a metaphoric chinese restaurant, the first customer seats at the first table. For $n> 1$, the $n$th customer seats at the $k$th (non-empty) table with probability $\frac{c_{n,k}}{n-1+\theta}$ (where $c_{n,k}$ is the number of customers seated at this table), or at an empty table with probability $\frac{\theta}{n-1+\theta}$. Let $X_n$ denote the number of non-empty tables after the arrival of the $n$th customer. For exemple, let us sample the first  $50$ steps of the process, for $\theta=1$ :\\

\begin{center}
   \includegraphics[width=0.9\linewidth]{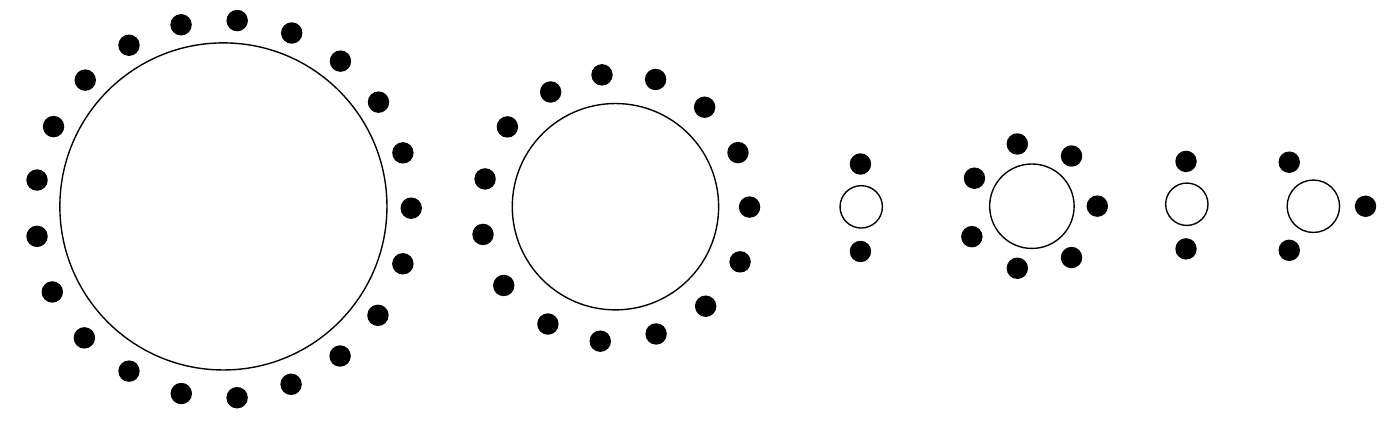}\\[0.3cm]
Figure 11: A realization of the chinese restaurant process (here $X_{50}=6$).\\
\end{center}

\begin{proof}[Proof of Proposition \ref{chine}] In this example, $(Y_i)_{i \ge  1}$ is  a family of independent Bernoulli random variables  with respective parameters $p_i=\theta/(i-1+\theta)$. We have :
\begin{align*}
T(n,k)&={n\brack k},\quad T_n(\theta)=\sum_kT(n,k)\theta^{k}=(\theta)^{\uparrow n},\quad c_n(\theta)=\frac1{\theta+n}.
\end{align*}
Thus the probability distribution of $X_n$ is given, for $n\ge 1$, by:
\begin{align*}
h_n(\ell)&=\mathbb P(X_n=\ell)=\frac{\theta^\ell}{(\theta)^{\uparrow n}}{n\brack \ell}\ \ind_{1\le \ell\le n},
\end{align*}
see \cite[Section 3.1.3]{MR2245368}. For instance,
\begin{align*}
Q_{n,k,k+1}&=\frac{\theta}{n+\theta}=c_n(\theta)\ a(n+1,k)\ \theta .
\end{align*}
and
\begin{align*} 
p_{1}(n,k)&=P_{(n,k),(n-1,k-1)}=\frac{h_{n-1}(k-1)Q_{n-1,k-1,k}}{h_{n}(k)}
\\
&= \frac{ \frac{\theta^{k-1}}{(\theta)^{n-1\uparrow }}{n-1\brack k-1}\times \frac{\theta}{n-1+\theta}	}{{n\brack k}\frac{\theta^{k}}{(\theta)^{n\uparrow}}} = \frac{ {n-1\brack k-1}	}{{n\brack k}},
\end{align*}
as expected.                                                                                              
\end{proof}

\subsection{One-dimensional Internal Diffusion Limited Aggregation process}
\thispagestyle{empty}
Diaconis and Fulton \cite{MR1218674}  introduced the internal Diffusion Limited Aggregation process (iDLA). Lawler, Bramson and Griffeath \cite{zbMATH00120275} coined the terminology \emph{iDLA}, and obtained an asymptotic shape behaviour. In the iDLA process, an aggregate of particles on $\mathbb Z^d$  is built as follows: 
\begin{enumerate}
    \item the first particle settles at the origin;
    \item the next particles perform a symmetric random walk on $\mathbb Z^d$, starting from the origin, and  settle at the first empty site they encounter.
\end{enumerate}

\begin{figure}[H]
              \includegraphics[width=\linewidth]{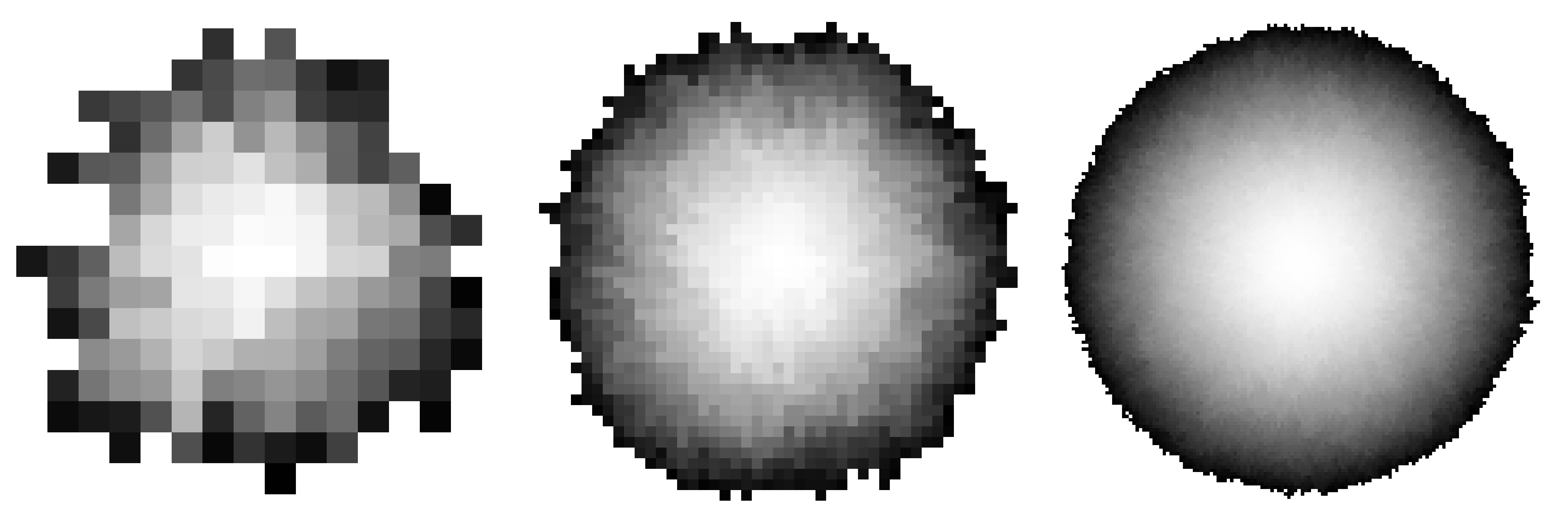}\\
              
              Figure 12: Normalized iDLA aggregates with 150, 1,500 and 15,000 particles on $\mathbb Z^2$.
\end{figure}

When $d=1$, let $X_{n}$ denote the number of particles settled to the right of the origin after the  $n$th step. Then, according to  \cite{zbMATH07227008}, the process $(X_{n})_n$ is an inhomogeneous Markov chain with the same distribution as the sequence of number of descents of the sequence of random permutations defined previously.  Both processes have the  one-dimensional distribution below
$$\mathbb P(X_n=k)=h_n(k)=\frac{\euler{n}{k}}{n!} \ind_{(n,k)\in S}.$$

In the case of the one-dimensional iDLA we can stack successive aggregates upon one another to form a space-time diagram. As with Figure 12, the longer it took to visit a cell, the darker we color it.

\begin{center}
    \begin{figure}[H]
    \centering
              \includegraphics[width=0.55\linewidth]{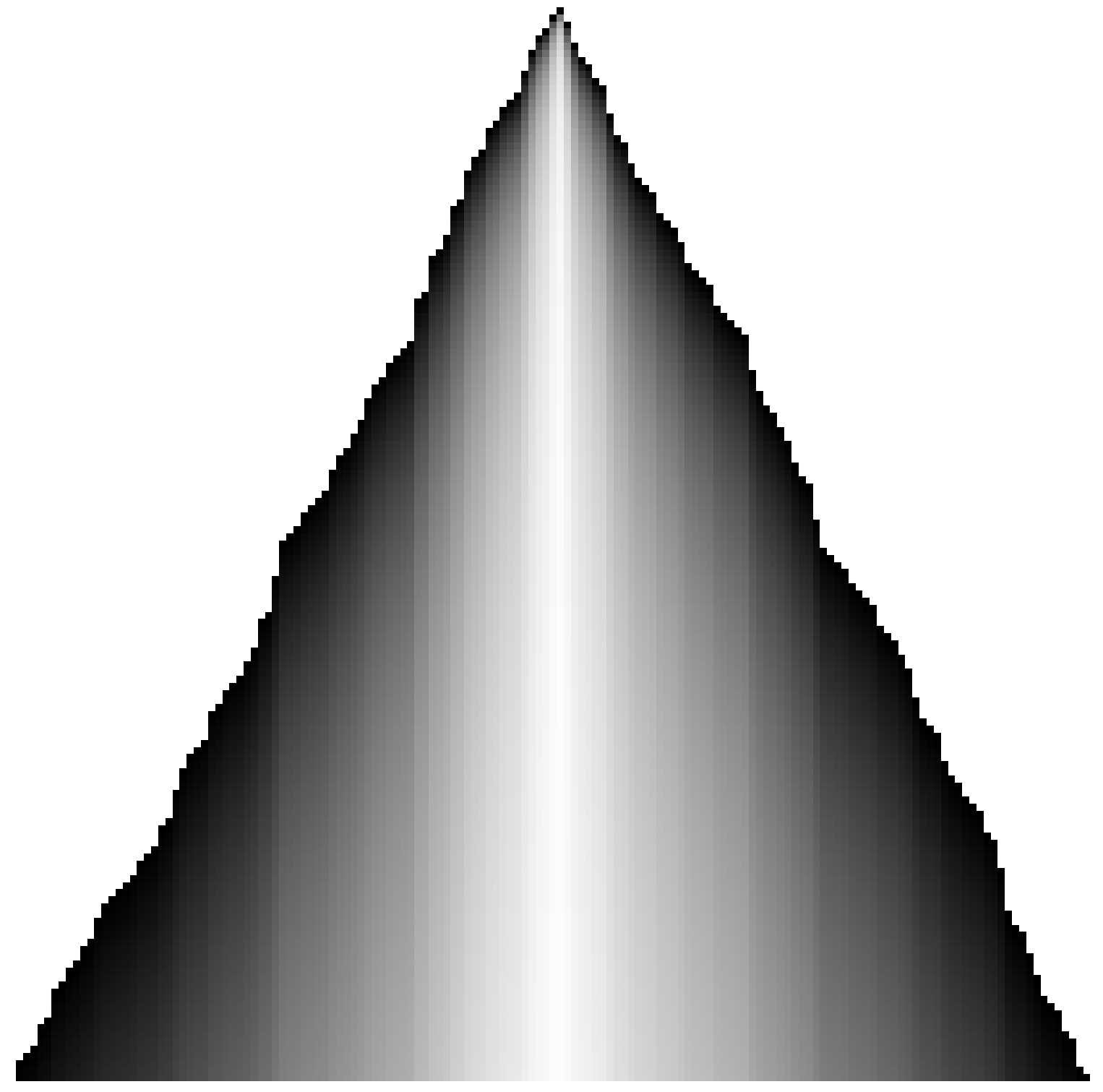}\\[0.5cm]
              Figure 13: Space-time diagram of\\ a one-dimensional iDLA.\\[-1cm]
    \end{figure}
\end{center}

\begin{proof}[Proof of Proposition \ref{idla}] In this example, we have :
\begin{align*}
T(n,k)&=\euler{n}{k},\quad T_n(1)=\sum_kT(n,k)=n!,\quad c_n(1)=\frac1{n}.
\end{align*}
Thus the probability distribution of $X_n$ is given, for $n\ge 1$, by:
\begin{align*}
h_n(\ell)&=\mathbb P(X_n=\ell)=\frac{1}{n!}\ \euler{n}{\ell}\ \ind_{1\le \ell\le n},
\end{align*}
see \cite[Section 3.1.3]{MR2245368}. For instance,
\begin{align*}
Q_{n,k,k+1}&=\frac{n-k}{n+1}=c_n(1)\ a(n+1,k).
\end{align*}
and
\begin{align*} 
p_{1}(n,k)&=P_{(n,k),(n-1,k-1)}=\frac{h_{n-1}(k-1)Q_{n-1,k-1,k}}{h_{n}(k)}
\\
&= \frac{ \frac{1}{n-1!}\euler{n-1}{k-1}\times\frac{n-k}{n}}{\euler{n}{k}\frac{1}{n!}} = \frac{ \euler{n-1}{k-1}(n-k)}{\euler{n}{k}},
\end{align*}
as expected.                                                                                              
\end{proof}

\thispagestyle{empty}

\section{The limit vector field : proof of Theorem \ref{slope}}
\label{powerform}
One can see the set $v$ of average jumps $v(n,k)$, defined, for $(n,k)\in S$, by
\begin{align*}
v(n,k)&=p_{0}(n,k)\times(-1,0)+p_{1}(n,k)\times(-1,-1)
\\
&=\pa{-1,-p_{1}(n,k)},
\end{align*}
as a kind of discrete vector field $v$ on $S$, with  slope  $p_{1}(n,k)$ at point $(n,k)$. As a consequence, the convergence of the sample paths of the time-reversed Markov chains of Section \ref{examplesMC} (see Theorem \ref{samplepath}) requires a precise asymptotic analysis of
\begin{align*}
p_{1}(n,k)&=\frac{a(n,k)T(n-1,k-1)}{T(n,k)},
\end{align*}
and thus, of $T(n,k)$. Consider the generating functions  $V_{k}$ and $H_{n}$ defined by
\begin{align*}
V_{k}(z)=\sum_{n=k}^{+\infty} \frac1{f_{n}} T(n,k) z^{n}, \quad H_{n}(w)&=\sum_{k=0}^{n} T(n,k) w^{k}, 
\end{align*}
respectively.  Here $f_n$ is either $1$ (for  Pascal's, resp. Euler's, triangle) or $n!$, for the 2 Stirling's triangles : for the enumeration of sets of labelled structures, such as  e.g. subsets or cycles, as for the 2 Stirling's triangles, cf. \cite[Part A]{MR2483235}, the factor $n!$ is due to the use of EGFs, and since we consider \emph{sets}, not sequences, of $k$ objects, i.e. \emph{unordered collections}, the  generating function $V_{k}$ contains a factor $1/k!$. In the first three cases, $V_{k}$ exhibits a factorisation $A\times B^{k}$ suitable for the saddle-point method, while, for eulerian numbers, $H_{n}$ is approximately of the form $B^{n}$, allowing the use of large deviations methods. 

Due to these factorisations, the limit vector field depends only on the slope $y/x$, and the function $\varphi$ depends on $B$ alone, in the first 3 cases through the saddle-point equation 
\begin{equation}
\label{spe}
\frac{B^{\prime}\pa{\impl}}{B\pa{\impl}}=\frac{1+\lambda}\impl,
\end{equation}
obtained by optimisation of the function $x\rightarrow\dfrac{B(x)}{x^{1+\lambda}}$ on $(0,+\infty)$, and, for eulerian numbers, through the Legendre transformation of $\ln B$, leading to the equation :
\begin{equation}
\label{lde}
\frac{B^{\prime}\pa{\impl}}{B\pa{\impl}}=\frac1{1+\lambda}.
\end{equation}

\begin{proof}[Proof of Theorem \ref{slope}.]
But for eulerian numbers, let $\impl(\lambda)$ be defined implicitly by \eqref{spe}, i.e. let $\impl(.)$ be the inverse function of :
$$x\longrightarrow \frac{xB^{\prime}\pa{x}}{B\pa{x}}-1.$$
The eulerian case is similar, but uses large deviations rather than saddle-point methods, and will be handled separately. In the remaining  3 cases, recall that $a(n,k)=1$, and  set : 
$$1+\lambda=\frac{n}{k},\quad \impl=\impl(\lambda),\quad 1+\tilde\lambda=\frac{n-1}{k-1},\quad \tilde\impl=\impl(\tilde\lambda).$$
For these 3 cases, the saddle-point method, see \cite[Part B, Chap. VIII]{MR2483235}, leads to 
\begin{align}
\label{homothetie}
T(n,k)\sim \frac{f_{n}}{f_{k}}\ \left(\frac{B\pa{\impl}}{\impl  ^{1+\lambda}}\right)^{k }\ g(n,k),
\end{align}
in which $g(.,.)$ is some factor such that $g(n,k)\sim g(n-1,k-1)$.  The invariance by homothetie of the field lines results from the factorisation $V_{k}=A\times B^{k}$ and  from the Cauchy formula, that leads to the key role of $\lambda$ in the asymptotic behaviour \eqref{homothetie}, and is thus a consequence of the decomposability of the underlying combinatorial structures. 

The factor $f_{n}/f_{k}$ matters only for the 2 Stirling's triangles. As a consequence, for the two Stirling triangles, we have
\begin{align*}
p_{1}(n,k)&\sim\frac{a(n,k)\ k }{n}\ \frac{\tilde\impl  ^{1+\tilde\lambda}}{B\pa{\tilde\impl  }}\left(\frac{B\pa{\tilde\impl  }}{\tilde\impl  ^{1+\tilde\lambda}}\frac{\impl  ^{1+\lambda}}{B\pa{\impl  }}\right)^{k}
\\
&\sim\frac{1}{1+\lambda}\ \frac{\impl  ^{1+\lambda}}{B\pa{\impl  }}\ \left(\frac{\tilde\impl  ^{1+\lambda}}{\tilde\impl  ^{1+\tilde\lambda}}\right)^{k}\left(\frac{B\pa{\tilde\impl  }}{\tilde\impl  ^{1+\lambda}}\frac{\impl  ^{1+\lambda}}{B\pa{\impl  }}\right)^{k}
\\
&\sim\frac{1}{1+\lambda}\ \frac{\impl  ^{1+\lambda}}{B\pa{\impl  }}\ \impl  ^{(\lambda-\tilde\lambda)k},
\end{align*}
the last step due to
\begin{align}
\label{cancelatn}
\lim_{k}\ k\ \ln\left(\frac{B\pa{\tilde\impl  }}{\tilde\impl  ^{1+\lambda}}\frac{\impl  ^{1+\lambda}}{B\pa{\impl  }}\right)&=0.
\end{align} 
Actually, since $\impl$ is solution of the saddle-point equation, the derivative of $$x\rightarrow\ln\left(\frac{B\pa{x}}{x^{1+\lambda}}\right)$$ vanishes at $\impl$, thus
$$\ln\left(\frac{B\pa{\tilde\impl  }}{\tilde\impl  ^{1+\lambda}}\frac{\impl  ^{1+\lambda}}{B\pa{\impl  }}\right)=o\left(\tilde\impl-\impl\right)=o\left(\lambda-\tilde\lambda\right),$$
but
$$\lambda-\tilde\lambda=\frac{n}{k}-\frac{n-1}{k-1}\sim\frac{-\lambda}{k},$$
entailing \eqref{cancelatn}. Thus
\begin{align*}
p_{1}(n,k)&\sim\frac{1}{1+\lambda}\ \frac{\impl  }{B\pa{\impl  }}.
\end{align*}
Finally, for Stirling's triangles, the saddle-point equation \eqref{spe} gives
\begin{align*}
p_{1}(n,k)&\sim\frac{1}{B^{\prime}\pa{\impl  }},
\end{align*}
For Pascal's triangle,   ${n\choose k}$ enumerates words with $n$ letters, $k$ among them being $\texttt{a}$'s and the $n-k$ others being $\texttt{b}$'s, thus Pascal's triangle enumerates \emph{sequences} (not sets) of \emph{unlabelled} objects\footnote{A word with $n$ letters, $k$ among them being $\texttt{a}$'s and the $n-k$ others being $\texttt{b}$'s, can be seen as a sequence of $k$ words of the form $\texttt{b}^{m}\texttt{a}$ followed by a word of the form $\texttt{b}^{m}$.}, for which one usually uses OGFs. As a consequence, $f_n=1$, and, compared with the previous computation, we are rid of the factor $n!/k!$ in $T(n,k)$, and of the factor $k/n=1/(1+\lambda)$ in $p_{1}(n,k)$, thus we obtain
\begin{align*}
p_{1}(n,k)&\sim \ \frac{\impl  }{B\pa{\impl  }}.
\end{align*}
Before we turn to the case of eulerian numbers, let us derive $\varphi$ for each of the 3 first cases :
\begin{itemize} 
\item
\emph{Pascal's triangle :} 
\begin{align*}
V_{k,1}(z)&=\sum_{n\ge k}{n\choose k}z^{n}=\ \dfrac{1}{1-z}\ \pa{\dfrac{z}{1-z}}^{k},
\\
B_{1}(z)&=\dfrac{z}{1-z},
\\
p_{1}(n,k)&\sim \frac{\impl  }{B_{1}\pa{\impl  }}=1-\impl.
\end{align*}
Here \eqref{spe} can be written
$$\frac1{1-\impl}=1+\lambda,$$
thus $\varphi_{1}(k/n)=\frac1{1+\lambda}=k/n$, that is :
$$p_{1}(n,k)\sim\frac kn,$$
which is not a surprise, since it is well known that, actually, $p_{1}(n,k)=\frac kn$.
\item \emph{Stirling numbers of the second kind}

Here :
\begin{align*}
V_{k,2}(z)&=\sum_{n\ge k}{n\brace k}\dfrac{z^{n}}{n!}=\ \dfrac{1}{k !}\ \pa{e^{z}-1}^{k},
\\
B_{2}(z)&=e^{z}-1,
\\
p_{1}(n,k)&\sim \frac{1}{B^{\prime}_{2}\pa{\impl  }}=e^{-\impl},
\end{align*}
and \eqref{spe} can be written
\begin{align}
	\label{poissonpos}
\frac{\impl}{1-e^{-\impl}}=1+\lambda, \end{align}
see \cite{amri:hal-02164935}. Thus
$$\varphi_{2}(k/n)=e^{-\impl\pa{\tfrac nk-1}}.$$
Note that, according to Good \cite{MR0120204} and others, $\impl$ is a smooth concave function of $\lambda>0$, with positive values. Note also that \eqref{poissonpos} is the equation to be solved when one wants to tune the parameter $\zeta$ of a Poisson random variable \emph{conditioned to be positive} in order to obtain the expectation $1+\lambda$.

\item \emph{Stirling numbers of the first kind (unsigned)} 
\begin{align*}
V_{k,3}(z)&=\sum_{n\ge k}{n\brack k}\dfrac{z^{n}}{n!}=\ \dfrac{1}{k !}\ \pa{-\ln(1-z)}^{k},
\\
B_{3}(z)&=-\ln(1-z),
\\
p_{1}(n,k)&\sim \frac{1}{B^{\prime}_{3}\pa{\impl  }}=1-\impl.
\end{align*}
Here \eqref{spe} can be written
\begin{align}
\label{fisherlog}
\frac{\impl}{\pa{\impl-1}\ln(1-\impl)}=1+\lambda,
\end{align}
which defines $\impl$ as smooth concave function of $\lambda>0$, with values in $(0,1)$.
Thus
\begin{align}
\label{fisherlog2}
\varphi_{3}(k/n)=1-\impl\pa{\tfrac nk-1}.
\end{align}
Note that \eqref{fisherlog} is the equation to be solved when one wants to tune the parameter $\zeta$ of a logarithmic probability distribution in order to obtain the expectation $1+\lambda$.
\end{itemize}

For \emph{eulerian numbers}, though the computation of $\varphi_{4}$ has a similar flavour, it presents some notable differences. In order to sum up the asymptotic analysis of eulerian numbers, set, as done in \cite{MR375433} : $$t=\frac{k}n=\frac1{1+\lambda}.$$
In \cite[page 97]{MR375433}, the  main tool is the approximation of $H_{n}(e^{s})$, the Laplace transform of $h_{n}$, by $$B(s)^{n+1}=r(s)^{-n-1}=\pa{\frac{e^{s}-1}s}^{n+1}.$$ In other terms, the key point in  \cite{MR375433} is that $h_{n}$ is approximately the distribution of the sum of $n+1$ i.i.d. uniform random variables, with Laplace transform $B(s)$. This is reminiscent of Tanny's representation of eulerian numbers (cf. \cite{zbMATH03445253}) :
\begin{align}
\label{Tanny}
h_{n}\left(k\right)= \dfrac{\euler{n}{k}}{n!}
&=\mathbb P\left(\lfloor U_{1}+U_{2}+\dots+U_{n}\rfloor=k\right).
\end{align}
Bender  obtains the following asymptotic formula for $\euler{n}{k}$ when $(n,k)$ goes to infinity
$$\frac{\euler{n}{k}}{n!}\sim \pa{B(\impl)e^{-\impl t}}^{n}\ g(n,k)$$
in which $g(.,.)$ is some factor such that $g(n,k)\sim g(n-1,k-1)$, and in which
  $\impl$ is the only real number such that 
\begin{align}
\nonumber 
\frac1{1+\lambda}=t&=\frac{\partial}{\partial \impl}\ln\pa{\frac{e^{\impl}-1}{\impl}}
\\&=\frac{e^\impl}{e^\impl -1}-\frac{1}{\impl}.
\end{align}
One recognize in $\impl(.)$ the derivative of the Legendre-Fenchel transformation of the cumulant-generating function of the uniform distribution, i.e.
the unique solution of
\begin{align}
\label{spe2}
\frac{\partial}{\partial\impl}\ln\pa{B(\impl)e^{-\impl t}}&= \frac{B^{\prime}(\impl)}{B(\impl)}-t =0.
\end{align}

As a consequence, for eulerian numbers, we have
\begin{align*}
p_{1}(n,k)&\sim\frac{a(n,k)}{n}\ \pa{\frac{B(\tilde\impl)^{n-1}e^{-\tilde\impl (k-1)}}{B(\impl)^{n}e^{-\impl  k}}}
\\
&\sim\frac{a(n,k)}{n}\ \frac{e^{\tilde\impl }}{B(\tilde\impl)}\pa{\frac{B(\tilde\impl)^{n}e^{-\tilde\impl k}}{B(\impl)^{n}e^{-\impl  k}}}
\\
&\sim\frac{(1-t)e^{\impl }}{B(\impl)}\pa{\frac{B(\tilde\impl)e^{-\tilde\impl  t}}{B(\impl)e^{-\impl  t}}}^{n}
\\
&\sim\frac{(1-t)e^{\impl }}{B(\impl)}
\end{align*}
the last step due to
\begin{align}
\label{cancelatn2}
\lim_{n}\ n\ \ln\left(\frac{B(\tilde\impl)e^{-\tilde\impl  t}}{B(\impl)e^{-\impl  t}}\right)&=0.
\end{align} 
Actually, since $\impl$ is solution of \eqref{spe2}, the derivative of $$x\rightarrow\ln\left(B(x)e^{-x  t}\right)$$ vanishes at $\impl$, thus
$$\ln\left(\frac{B(\tilde\impl)e^{-\tilde\impl  t}}{B(\impl)e^{-\impl  t}}\right)=o\left(\tilde\impl-\impl\right)=o\left(t-\tilde t\right),$$
but
$$t-\tilde t=\frac{k}{n}-\frac{k-1}{n-1}\sim\frac{1-t}{n},$$
entailing \eqref{cancelatn2}. Thus
\begin{align*}
p_{1}(n,k)&\sim\frac{(1-t)\impl e^{\impl}}{e^{\impl}-1}=\frac{\lambda\impl }{(1+\lambda)(1-e^{-\impl})}=\phi_{4}(\lambda).
\end{align*}
Note that :
$$\impl(1-t)=-\impl(t),\quad\varphi_4(1-t)= 1-\varphi_4(t)= \frac{t \impl}{e^{\impl}-1},$$
as expected from the relation $\euler{n}{k}=\euler{n}{n-k-1}$.
\end{proof}

%
%
%


%
%
%
%

\section{Sample path convergence}
\label{saddle-point}
This section is devoted to the proof of Theorem \ref{samplepath} for the first three triangles.   For the sake of completeness, we first give the well known proof of Theorem \ref{samplepath} for Pascal's triangle.   In the case of Stirling numbers of the second kind, a weaker form of Theorem \ref{samplepath} was obtained in \cite{amri:hal-02164935} at the price of a tedious proof using Wormald method and saddle-point asymptotics. For Euler's triangle,  we think that the same property holds true, but the proof is still a work in progress. In the case of Stirling triangles of both kind, we believe that the proofs given in the next sections are new.

\subsection{Proof of Theorem \ref{samplepath} : Pascal's triangle.} Consider   two probability distributions for the processes $(W,X,Y)$ defined at section \ref{examplesMC}. Under $\mathbb{P}_{(m,\ell)}$, $W$ is a Markov chain starting from $(m,\ell)$, with transition probabilities $(p_{\varepsilon}(n,k))_{\varepsilon,n,k}$ related to Pascal's triangle, and the processes $(X,Y)$ are distributed accordingly. On the other hand, under $\mathbb{P}_{p}$, $Y=(Y_k)_{1\leqslant k \leqslant m}$ is a sequence of i.i.d Bernoulli random variables with parameter $p$, and the processes $(X,W)$ are distributed accordingly. According to Proposition \ref{kennedy}, for any $p\in(0,1)$, and any set $B$ in the relevant state space,
\begin{align}
\label{sufficient}
\mathbb{P}_{(m,\ell)} \pa{(W,X,Y) \in B}&= \mathbb{P}_{p} \pa{\{(W,X,Y) \in B\} \cap \{X_m=\ell\}}/ \mathbb{P}_{p} \pa{X_m=\ell}
\\
\nonumber&= \mathbb{P}_{p} \pa{\{(W,X,Y) \in B\} \cap \{X_m=\ell\}}/ \mathbb{P}_{p} \pa{W_0=(m,\ell)}.
\end{align}
By Hoeffding's inequality, for all $t>0$, and all $n\in[\![1,m]\!]$,
\begin{align*}\mathbb{P}_{p}(|X_n-np|\geqslant t)&\leqslant 2\exp\left(-\frac{2t^2}{n}\right)\\
&\leqslant 2\exp\left(-\frac{2t^2}{m}\right).
\end{align*}
In particular, for any $\eta\in (0,1/2)$ and for $t=m^{1-\eta}/\sqrt{2}$,
$$\mathbb{P}_{p}(|X_n-np|\geqslant m^{1-\eta}/\sqrt{2})\leqslant 2\exp\left(-m^{1-2\eta}\right).$$
Thus
\begin{align*}
\mathbb{P}_{p}(\exists n\in [\![0,m]\!]\text{ s.t. }|X_n-np|\geqslant m^{1-\eta} /\sqrt{2})&\leqslant \sum_{n=1}^m\mathbb{P}_{p}(|X_n-np|\geqslant m^{1-\eta} /\sqrt{2})
\\
&\leqslant 2m\exp\left(-m^{1-2\eta}\right).
\end{align*}
Set
\begin{align*}
A_m&=\ac{\exists n\in[\![0,m]\!] \text{~s.t.~}\|W_n-(m-n,(m-n)p)\|_1\geqslant m^{1-\eta} /\sqrt{2}}.
\end{align*}
Then $\mathbb{P}_{p}({A}_m \cap \{X_m=\ell\})\leqslant 2m\exp\left(-m^{1-2\eta}\right)$ and, according to \eqref{sufficient},
$$\mathbb{P}_{(m,\ell)}(A_m)=\frac{ \mathbb{P}_{p}({A}_m\cap \{X_m=\ell\})}{\mathbb{P}_{p}(X_m=\ell)}\leqslant \frac{2m\exp\left(-m^{1-2\eta}\right)}{ \mathbb{P}_{p}(X_m=\ell)}$$
This is true for any $p\in(0,1)$, thus for $p=\ell/m$ too, but, using Stirling formula, one finds
\begin{align*}
    \mathbb{P}_{\ell/m}(X_m=\ell)&=\binom{m}{\ell} \pa{\frac{\ell}{m}}^m \pa{\frac{m-\ell}{m}} ^{m-\ell}
    \\
    &\sim \frac{1}{\sqrt{2\pi p(1-p)}}\ \frac{1}{\sqrt{m}}.
\end{align*}
Finally, $\mathbb{P}_{(m,\ell)}(A_m)=\mathcal{O}(m^{3/2}e^{-m^{1-2\eta}})$ and vanishes for $\eta\in (0,1/2)$. For Pascal triangle, recall that $\gamma_{m,\ell}(t)= \ell t/m$, thus 
$$A_m=\ac{\sup\ac{\abs{w_{m}(t)-\gamma_{m,\ell}(t)},  mt\in[\![0,m]\!] }\ge m^{-\eta} /\sqrt{2}},$$
and
$$0\le \sup_{t\in[0,1]}\ac{\abs{w_{m}(t)-\gamma_{m,\ell}(t)}}-\sup_{mt\in[\![0,m]\!]}\ac{\abs{w_{m}(t)-\gamma_{m,\ell}(t)}}\le \frac\ell{m^2} \le \frac1{m},$$
so that, for $m$ large enough,
$$\ac{\sup_{t\in[0,1] }\abs{w_{m}(t)-\gamma_{m,\ell}(t)}\ge m^{-\eta}}\subset A_m,$$
and, as expected,
$$\lim_{m}\mathbb P_{(m,\ell)}\left(\sup_{t\in[0,1] }\abs{w_{m}(t)-\gamma_{m,\ell}(t)}\ge m^{-\eta}\right)=0.$$ 

\subsection{Proof of Theorem \ref{samplepath} : Stirling numbers of the first kind.} 

Consider   two probability distributions for the processes $(W,X,Y)$ defined at section \ref{examplesMC}. Under $\mathbb{P}_{(m,\ell)}$, $W$ is a Markov chain starting from $(m,\ell)$, with transition probabilities $(p_{\varepsilon}(n,k))_{\varepsilon,n,k}$ related to Stirling numbers \emph{of the first kind}, and the processes $(X,Y)$ are distributed accordingly. On the other hand, under $\mathbb{P}_{\theta}$, $(Y_i)_{i \ge  1}$ is  a family of independent Bernoulli random variables  with respective parameters $p_i=\theta/(i-1+\theta)$, and the processes $(W,X,Y)$ are distributed accordingly :  for instance, $X_n$ can be seen as the number of non-empty tables after the arrival of the $n$th customer, as in Section \ref{chinese}. As before, according to Proposition \ref{chine}, for any $\theta>0$, and any set $B$ in the relevant state space,
\begin{align}
\label{sufficientchine}
\mathbb{P}_{(m,\ell)} \pa{(W,X,Y) \in B}&= \mathbb{P}_{\theta} \pa{\{(W,X,Y) \in B\} \cap \{X_m=\ell\}}/ \mathbb{P}_{\theta} \pa{X_m=\ell}
\\
\nonumber&= \mathbb{P}_{\theta} \pa{\{(W,X,Y) \in B\} \cap \{X_m=\ell\}}/ \mathbb{P}_{\theta} \pa{W_0=(m,\ell)}.
\end{align}
Also, as in 	\eqref{fisherlog}, recall that
$$\frac{\impl}{\pa{\impl-1}\ln(1-\impl)}=\frac{m}\ell=1+\lambda,$$
and that, for $t\ge 0$,
\begin{equation*}
\gamma_{m,\ell}(t)=\frac{1-\impl}{\impl}\ln\pa{\frac{1-\impl+t\,\impl}{1-\impl}}.
\end{equation*}
Let $\mu=\mu_{\theta,n}$ denote the expectation of $X_n$ under $\mathbb{P}_{\theta}$, that is
$$\mu=\mu_{\theta,n}=\sum_{k=1}^n\frac{\theta}{k-1+\theta},$$
and note that, for the choice $\theta_m=m(1-\impl)/\impl$,
\begin{align}
	\label{harmochine}
	\abs{\mu_{\theta_m,n}- m\,\gamma_{m,\ell}(n/m)}&\le \frac{n\impl}{m(1-\impl)} \le \frac{\impl}{1-\impl}.
\end{align}
According to Hoeffding's inequality, for all $t>0$, and all $n\in[\![1,m]\!]$,
\begin{align*}\mathbb{P}_{\theta}(|X_n-\mu_{\theta,n}|\geqslant t)&\leqslant 2\exp\left(-2t^2/n\right)\\
&\leqslant 2\exp\left(-2t^2/m\right).
\end{align*}
In particular, for all $\eta\in (0,1/2)$ and for $t=m^{1-\eta}/\sqrt{2}$,
$$\mathbb{P}_{\theta}(|X_n-\mu_{\theta,n}|\geqslant m^{1-\eta} /\sqrt{2})\leqslant 2\exp\left(-m^{1-2\eta}\right).$$
Set
\begin{align*}
A_m&=\ac{\exists n\in[\![0,m]\!] \text{~s.t.~}|X_n-\mu_{\theta,n} |\geqslant m^{1-\eta} /\sqrt{2}}.
\end{align*}
Thus
\begin{align*}
\mathbb{P}_{\theta}(A_m)&\leqslant \sum_{n=1}^m\mathbb{P}_{\theta}(|X_n-\mu_{\theta,n} |\geqslant m^{1-\eta} /\sqrt{2})
\\
&\leqslant 2m\exp\left(-m^{1-2\eta}\right)
\end{align*}
Then $\mathbb{P}_{\theta}({A}_m \cap \{X_m=\ell\})\leqslant 2m\exp\left(-2m^{1-2\eta}\right)$ and, according to \eqref{sufficientchine},
$$\mathbb{P}_{(m,\ell)}(A_m)=\frac{ \mathbb{P}_{\theta}({A}_m\cap \{X_m=\ell\})}{\mathbb{P}_{\theta}(X_m=\ell)}\leqslant \frac{2m\exp\left(-m^{1-2\eta}\right)}{ \mathbb{P}_{\theta}(X_m=\ell)}$$
This is true for any $\theta>0$, thus for $\theta_m=(1-\impl)m/\impl$ too, but, using relation (13) in \cite{MR0120204}, one finds
\begin{align*}
    \mathbb{P}_{\theta_m}(X_m=\ell)&= \frac{\theta_m^\ell}{(\theta_m)^{\uparrow m}}{m\brack \ell}\ \ind_{1\le \ell\le m},
    \\
    &\sim \frac{1}{\sqrt{m}}\ \sqrt{\frac{\ln(1-\impl)}{2\pi(1+\lambda)(\impl+\ln(1-\impl))}}.
\end{align*}
Finally, $\mathbb{P}_{(m,\ell)}(A_m)=\mathcal{O}(m^{3/2}e^{- m^{1-2\eta}})$ and vanishes for $\eta\in (0,1/2)$. But,  for $m$ large enough, 
$$B_m=\ac{\sup_{t\in[0,1] }\abs{w_{m}(t)-\gamma_{m,\ell}(t)}\ge m^{-\eta}}\subset A_m,$$
and, as expected,
$$\lim_{m}\mathbb P_{(m,\ell)}\left(B_m\right)=0.$$ 
Actually, due to \eqref{harmochine}, for $0\le n\le m$,
\begin{align*}	\abs{\frac{\mu_{N_0,n}}{m}-\gamma_{m,\ell}(n/m)}&\le \frac{\impl}{m(1-\impl)},
\end{align*}
and 
$$0\le \sup_{t\in[0,1]}\ac{\abs{w_{m}(t)-\gamma_{m,\ell}(t)}}-\sup_{mt\in[\![0,m]\!]}\ac{\abs{w_{m}(t)-\gamma_{m,\ell}(t)}}\le \frac1{m},$$
thus $B_m\subset A_m$  provided that
$$\frac{m^{-\eta}}{\sqrt{2}}+ \frac{\impl}{m(1-\impl)} +\frac{1}m \le m^{-\eta}.$$

\subsection{Proof of Theorem \ref{samplepath} : Stirling numbers of the second kind.} 
\label{stir2th2}

Consider   two probability distributions for the processes $(W,X,Y)$ defined at section \ref{examplesMC}. Under $\mathbb{P}_{(m,\ell)}$, $W$ is a Markov chain starting from $(m,\ell)$, with transition probabilities $(p_{\varepsilon}(n,k))_{\varepsilon,n,k}$ related to Stirling numbers \emph{of the second kind}, and the processes $(X,Y)$ are distributed accordingly. On the other hand, under $\mathbb{P}_{N}$, $X_n$ is the number of different coupons that have been collected after $n$ draws with replacement in a collection of $N$ available coupons, and the processes $(X,Y,W)$ are distributed accordingly. According to Proposition \ref{coupon}, for any $N\ge\ell$, and any set $B$ in the relevant state space,
\begin{align}
\label{sufficientcoupon}
\mathbb{P}_{(m,\ell)} \pa{(W,X,Y) \in B}&= \mathbb{P}_{N} \pa{\{(W,X,Y) \in B\} \cap \{X_m=\ell\}}/ \mathbb{P}_{N} \pa{X_m=\ell}
\\
\nonumber&= \mathbb{P}_{N} \pa{\{(W,X,Y) \in B\} \cap \{X_m=\ell\}}/ \mathbb{P}_{N} \pa{W_0=(m,\ell)}.
\end{align}
Let $\mu$ denote the expectation of $X_n$ under $\mathbb{P}_{N}$, that is
$$\mu=\mu_{N,n}=N\pa{1-\pa{1-\frac{1}{N}}^n},$$
and note that 
\begin{align}
	\label{expo}
	\abs{\mu_{N,n}-N\pa{1-e^{-\frac{n}{N}}}}&\le \frac{n}{2N}.
\end{align}
Also, as in 	\eqref{poissonpos}, set
\begin{align*}
\frac{\impl}{1-e^{-\impl}}=1+\lambda= \frac{m}{\ell}. \end{align*}
According to \cite[Ch. 4, Theorem 4.18]{zbMATH00819814}, by Azuma-Hoeffding's inequality, for all $t>0$, and all $n\in[\![1,m]\!]$,
\begin{align*}\mathbb{P}_{N}(|X_n-\mu_{N,n}|\geqslant t)&\leqslant 2\exp\left(-t^2\ \frac{N-1/2}{N^2-\mu_{N,n}^2}\right)\\
&\leqslant 2\exp\left(-\frac{t^2}{2N}\right).
\end{align*}
In particular, for all $\eta\in (0,1/2)$ and for $t=m^{1-\eta}/\sqrt{2}$,
$$\mathbb{P}_{N}(|X_n-\mu_{N,n}|\geqslant m^{1-\eta} /\sqrt{2})\leqslant 2\exp\left(-m^{2-2\eta}/4N\right).$$
Set
\begin{align*}
A_m&=\ac{\exists n\in[\![0,m]\!] \text{~s.t.~}|X_n-\mu_{N,n} |\geqslant m^{1-\eta} /\sqrt{2}}.
\end{align*}
Thus
\begin{align*}
\mathbb{P}_{N}(A_m)&\leqslant \sum_{n=1}^m\mathbb{P}_{N}(|X_n-\mu_{N,n} |\geqslant m^{1-\eta} /\sqrt{2})
\\
&\leqslant 2m\exp\left(-m^{2-2\eta}/4N\right)
\end{align*}
Then $\mathbb{P}_{N}({A}_m \cap \{X_m=\ell\})\leqslant 2m\exp\left(-m^{2-2\eta}/4N\right)$ and, according to \eqref{sufficientcoupon},
$$\mathbb{P}_{(m,\ell)}(A_m)=\frac{ \mathbb{P}_{N}({A}_m\cap \{X_m=\ell\})}{\mathbb{P}_{N}(X_m=\ell)}\leqslant \frac{2m\exp\left(-m^{2-2\eta}/4N\right)}{ \mathbb{P}_{N}(X_m=\ell)}$$
This is true for any $N\ge\ell$, thus for $N_0=\lceil m/\impl\rceil$ too, but, using relation (3) in \cite{MR0120204}, one finds
\begin{align*}
    \mathbb{P}_{\lceil m/\impl\rceil}(X_m=\ell)&=\frac{N_0!\,N_0^{-m}}{N_0-\ell!}\,{m\brace\ell}
    \\
    &\sim \sqrt{\frac{\impl\ e^{\impl}}{2\pi (\impl-\lambda)m}}.
\end{align*}
Finally, $\mathbb{P}_{(m,\ell)}(A_m)=\mathcal{O}(m^{3/2}e^{-\impl m^{1-2\eta}})$ and vanishes for $\eta\in (0,1/2)$. For Stirling numbers of the second kind, recall that $\gamma_{m,\ell}(t)= \pa{1-e^{-\impl t}}/\impl$, thus 
$$B_m=\ac{\sup\ac{\abs{w_{m}(t)-\gamma_{m,\ell}(t)},  mt\in[\![0,m]\!] }\ge m^{-\eta}}\subset A_m.$$
Actually, due to \eqref{expo}, for $0\le n\le m$,
\begin{align}
\nonumber	\abs{\frac{\mu_{N_0,n}}{m}-\gamma_{m,\ell}(n/m)}&\le \frac{n}{2N_0m}+\abs{\frac{N_0}{m}\pa{1-e^{-\frac{n}{N_0}}}-\gamma_{m,\ell}(n/m)}
	\\
\nonumber	&\le \frac{n}{2N_0m}+\abs{\frac{1-e^{-\tilde\impl n/m}}{\tilde\impl}-\frac{1-e^{-\impl n/m}}{\impl}}
	\\
	\label{lipsch}
	&\le \frac{\tilde\impl}{2m}+\frac{1+ \tilde\impl}m,
\end{align}
in which 
$$\frac{m}{\tilde\impl}=\left\lceil\frac{m}{\impl}\right\rceil =N_0, \quad\text{~thus~} \quad 0\le\impl-\tilde\impl\le\frac{\impl\tilde\impl}m.$$
Thus $B_m\subset A_m$ for $m$ large enough, i.e. provided that
$$\frac{m^{-\eta}}{\sqrt{2}}+ \frac{\tilde\impl}{2m}+\frac{1+ \tilde\impl}m \le m^{-\eta}.$$
Finally 
$$0\le \sup_{t\in[0,1]}\ac{\abs{w_{m}(t)-\gamma_{m,\ell}(t)}}-\sup_{mt\in[\![0,m]\!]}\ac{\abs{w_{m}(t)-\gamma_{m,\ell}(t)}}\le \frac1{m},$$
so that, for $m$ large enough,
\begin{align}
\label{include2}
\ac{\sup_{t\in[0,1] }\abs{w_{m}(t)-\gamma_{m,\ell}(t)}\ge m^{-\eta}}&\subset A_m,	
\end{align}
and, as expected,
$$\lim_{m}\mathbb P_{(m,\ell)}\left(\sup_{t\in[0,1] }\abs{w_{m}(t)-\gamma_{m,\ell}(t)}\ge m^{-\eta}\right)=0.$$ 

\subsection{Application to the enumeration of accessible complete deterministic automata with $k$ letters and $n$ vertices}
Let $a_{k,n}$ denote the number of accessible complete deterministic automata (ACDA) with $k$ letters and $n$ vertices (see \cite{Nicaudth,amri:hal-02164935} for definitions). According to Kor\v{s}unov \cite{MR517814,MR862029}, for any given $k\ge 2$, 
\begin{align}
\label{kor}
a_{k,n}&\sim c_{k}\ {kn+1\brace n} n!,
\end{align}
in which $\impl_2$ is defined by \eqref{ckor2}, and
\begin{align}
\label{ckor}
c_{k}&= 1-k\,e^{-\impl_2(k-1)}.
\end{align}
Following \cite{amri:hal-02164935}, this section gives a probabilistic interpretation of  Kor\v{s}unov's formula, that relies on Theorem \ref{samplepath} for Stirling numbers of the second kind : according to \cite{Nicaudth}, there exists a bijection between the set of ACDA with $k$ letters and $n$ vertices and a subset $\mathcal{A}_{k,n}$ of the set $\Omega_{kn+1,n}$ of surjections from $[\![kn+1]\!]$ to $[\![n]\!]$. Thus \eqref{kor} states that the ratio $\#\mathcal{A}_{k,n}/\#\Omega_{kn+1,n}$ converges to $c_k$ with $n$. But an element of $\Omega_{kn+1,n}$ can be seen as the sample path of a coupon collector process such that the collection of $n$ items is complete at step $kn+1$. As a consequence, in the notations of Section \ref{stir2th2},
\begin{equation}
	\label{interpretation}
	\mathbb{P}_{(kn+1,n)}(\mathcal{A}_{k,n})=  \frac{\#\mathcal{A}_{k,n}}{\#\Omega_{kn+1,n}}=\frac{a_{k,n}}{{kn+1\brace n}n!},
\end{equation}
and Kor\v{s}unov's formula can be rephrased as
\begin{equation}
	\label{interpretation2}
	\lim_n\mathbb{P}_{(kn+1,n)}(\mathcal{A}_{k,n})=  c_k.
\end{equation}
Now, according to \cite{Nicaudth}, $\mathcal{A}_{k,n}$ is the set of elements $\omega\in\Omega_{kn+1,n}$ such that 
\begin{equation*}
	\forall \ell\in[\![0,n-1]\!],\quad X_{\ell k+1}(\omega)\ge \ell+1,
\end{equation*}
or, equivalently,
\begin{equation}
\label{connexion}
	\forall \ell\in[\![0,kn]\!],\quad k\,X_{\ell}(\omega)\ge \ell.
	\end{equation}
Relation \eqref{connexion} is, as usual, required from a breadth first search walk to insure the connexity of the underlying graph. 

We shall now sketch the argument, taken from  \cite{amri:hal-02164935}, which, using Theorem \ref{samplepath}, shows that $\Upsilon_n=\overline{\mathcal{A}_{k,n}}$ satisfies
\begin{equation*}
	\lim_n\mathbb{P}_{(kn+1,n)}(\Upsilon_n)=  1-c_k= k\,e^{-\impl_2(k-1)}.
\end{equation*}
Note that $\lim_n \lambda(kn+1,n)=k-1$, and that the corresponding concave limit field line,
$$\gamma_{m,\ell}(t)= \frac{1-e^{-\impl_2(k-1) t}}{\impl_2(k-1)},$$
crosses the line $y=x/k$ only  at its endpoints $(0,0)$ and $(k,1)$, so that, according to Theorem \ref{samplepath}, but for an exponentially small probability, the sample path $\ac{(\ell, X_\ell), 0\le\ell\le kn+1}$
crosses the line $y=x/k$ only  close to its endpoints. The probability that such a crossing occurs close to $(0,0)$ is very small too, see \cite[Proposition 2]{amri:hal-02164935}. As a consequence, $\mathbb{P}_{(kn+1,n)}(\Upsilon_n)$ has the same asymptotic behaviour than the probability that the sample path $\ac{(\ell, X_\ell), 0\le\ell\le kn+1}$
crosses the line $y=x/k$   close to its endpoint $(kn+1,n)$.  Close to this endpoint, the sample path has approximately the same transition probabilities as a standard random walk with step distribution
\[\pa{1-e^{-\impl_2(k-1)}}\delta_{0}+e^{-\impl_2(k-1)}\delta_{-1}.\] 
The probability of a crossing of the line $y=x/k$ by such a standard random walk is $1-c_k= k\,e^{-\impl_2(k-1)}$, as follows for instance from the Pollaczek-Khinchine formula (cf. Corollary 6.6 of \cite{MR1978607}, or Proposition 3 of \cite{amri:hal-02164935}, in which one can find a proof of Kor\v{s}unov's formula along these lines).  

\section*{Acknowledgements.} The authors are grateful to Antoine Lejay and R\'emi Peyre for fruitful discussions, and to Antoine Lejay for pointing reference \cite{zbMATH03019588}.
\bibliographystyle{amsalpha}
\bibliography{triangle}

\end{document}